\documentclass[12pt]{amsart}  

\usepackage[latin1]{inputenc}
\usepackage{amsmath} 
\usepackage{amsfonts}
\usepackage{amssymb}
\usepackage{stmaryrd}
\usepackage{latexsym} 
\usepackage{graphicx}
\usepackage{subfigure}
\usepackage{color}
\usepackage{hyperref}
\usepackage{verbatim}
\usepackage[all]{xy}
\usepackage{graphics}
\usepackage{pdfsync}
\usepackage{tikz}
\usepackage{url}

\theoremstyle{plain}
\newtheorem{theorem}{Theorem}
\newtheorem{proposition}[theorem]{Proposition}

\newtheorem{lemma}[theorem]{Lemma}

\newtheorem{definition}[theorem]{Definition}
\newtheorem{example}[theorem]{Example}

\theoremstyle{remark}

\numberwithin{equation}{section}
\newcommand{\bQ}{\mathbb{Q}}

\newcommand{\suchthat}{\;|\;}
\newcommand{\spam}{\operatorname{span}}

\newlength\cellsize 
\savebox2{%
\begin{picture}(15,15)
\put(0,0){\line(1,0){15}}
\put(0,0){\line(0,1){15}}
\put(15,0){\line(0,1){15}}
\put(0,15){\line(1,0){15}}
\end{picture}}
\newcommand\cellify[1]{\def\thearg{#1}\def\nothing{}%
\ifx\thearg\nothing
\vrule width0pt height\cellsize depth0pt\else
\hbox to 0pt{\usebox2\hss}\fi%
\vbox to 15\unitlength{
\vss
\hbox to 15\unitlength{\hss$#1$\hss}
\vss}}
\newcommand\tableau[1]{\vtop{\let\\=\cr
\setlength\baselineskip{-16000pt}
\setlength\lineskiplimit{16000pt}
\setlength\lineskip{0pt}
\halign{&\cellify{##}\cr#1\crcr}}}
\savebox3{%
\begin{picture}(15,15)
\put(0,0){\line(1,0){15}}
\put(0,0){\line(0,1){15}}
\put(15,0){\line(0,1){15}}
\put(0,15){\line(1,0){15}}
\end{picture}}
\newcommand\expath[1]{%
\hbox to 0pt{\usebox3\hss}%
\vbox to 15\unitlength{
\vss
\hbox to 15\unitlength{\hss$#1$\hss}
\vss}}
\newcommand\bas[1]{\omit \vbox to \cellsize{ \vss \hbox to \cellsize{\hss$#1$\hss} \vss}}

\begin{document}

\title[Lollipop and lariat symmetric functions]{Lollipop and lariat symmetric functions}

\author{Samantha Dahlberg}
\address{
Department of Mathematics,
 University of British Columbia,
 Vancouver BC V6T 1Z2, Canada}
\email{samadahl@math.ubc.ca}

\author{Stephanie van Willigenburg}
\address{
 Department of Mathematics,
 University of British Columbia,
 Vancouver BC V6T 1Z2, Canada}
\email{steph@math.ubc.ca}

\thanks{
Both authors were supported  in part by the National Sciences and Engineering Research Council of Canada.}

\begin{abstract}
 We compute an explicit $e$-positive formula for the chromatic symmetric function of a lollipop graph, $L_{m,n}$. From here we deduce that there exist countably infinite distinct $e$-positive, and hence Schur-positive, bases of the algebra of symmetric functions whose generators are chromatic symmetric functions. Finally, we resolve 6 conjectures on the chromatic symmetric function of a lariat graph, $L_{n+3}$.
\end{abstract}

\maketitle

\section{Introduction}\label{sec:intro} A generalization of the chromatic polynomial of a graph that has recently been garnering much attention is the chromatic symmetric function \cite{Stan95}. As expected from the name, it generalizes many properties of the chromatic polynomial such as the number of acyclic orientations \cite[Theorem 3.3]{Stan95} though not the property of deletion-contraction. However, recently  a  triple-deletion property has been established \cite[Theorem 3.1]{Orellana} during the study of graphs with equal chromatic symmetric function, which we generalize to $k$-deletion in Proposition~\ref{prop:kdel}.  Another property in which the chromatic symmetric function potentially differs from the chromatic polynomial is that it may distinguish nonisomorphic trees \cite[p 170]{Stan95} and this has been one avenue of research \cite{Jose2+1, Jose2, MMW, Orellana}. Due to connections to other areas such as representation theory and algebraic geometry, another active avenue of research is the positivity of the chromatic symmetric function of a given graph when expanded into elementary symmetric functions, known as  \emph{$e$-positivity}  \cite{ChoHuh, Gash, GebSag, GP,    MM, Wolfe}, or when expanded into Schur functions, known as \emph{Schur-positivity} \cite{Gasharov, SW,  SWW}. Graphs of particular interest are indifference graphs due to their relation to imminants of  Jacobi-Trudi matrices \cite{StanStem}. One specific well-known family of indifference graphs are lollipop graphs, which are important in the study of random walks as they are graphs that achieve the maximum possible hitting time \cite{BW}, cover time \cite{Feige} and commute time \cite{Jonasson}. In \cite{GebSag} the chromatic symmetric function of lollipop graphs were confirmed to be $e$-positive indirectly, however the question still remains of an explicit combinatorial $e$-positive formula. 

In this paper we directly prove the $e$-positivity of the chromatic symmetric function of lollipop graphs in Theorem~\ref{the:epos}, via a recurrence relation in Theorem~\ref{the:X_{L_{m,n}}}, and consequently give an explicit combinatorial $e$-positive formula in Proposition~\ref{prop:lolliKm}. A useful technique is the aforementioned triple-deletion, which we generalize to $k$-deletion in Proposition~\ref{prop:kdel}. We subsequently prove the chromatic symmetric function distinguishes nonisomorphic lollipop graphs and then use  $e$-positivity to prove there exist countably infinite distinct $e$-positive, and hence Schur-positive, bases of the algebra of symmetric functions whose generators are chromatic symmetric functions of connected graphs in Theorem~\ref{the:infbases}. Finally in Section~\ref{sec:lariat} we resolve 6 conjectures dating from 1998 on the chromatic symmetric function of lariat graphs, which are special cases of lollipop graphs. More precisely, 5 of these conjectures are proved in Theorem~\ref{the:lariat} and we provide a counterexample to the sixth.

\section{Background}\label{sec:background} Before we prove our results, let us recall useful combinatorial concepts and the algebra of symmetric functions. A \emph{partition} $\lambda = (\lambda_1, \lambda_2, \ldots , \lambda_\ell)$ of $N$, denoted by $\lambda \vdash N$, is a list of positive integers whose \emph{parts} $\lambda_i$ satisfy $\lambda_1\geq \lambda _2\geq \cdots \geq\lambda_\ell >0$ and $\sum _{i=1}^\ell \lambda _i = N$. If $\lambda$ has exactly $a_i$ parts equal to $i$ for $1\leq i \leq N$ we often denote $\lambda$ by $\lambda = (1^{a_1}, 2^{a_2}, \ldots , N^{a_N})$, from which we can obtain the partition known as the \emph{transpose} of $\lambda$, denoted by $\lambda^t$ and given by $\lambda ^t = (a_1+\cdots + a_N, a_2+\cdots +a_N, \ldots , a_N)$ with zeros removed.

The algebra of symmetric functions is a subalgebra of $\bQ [[ x_1, x_2, \ldots ]]$ and can be defined as follows. We define the \emph{$i$-th elementary symmetric function} $e_i$ for $i\geq 1$ to be
$$e_i = \sum _{j_1<j_2<\cdots < j_i} x_{j_1}\cdots x_{j_i}$$and given a partition $\lambda = (\lambda _1, \lambda_2, \ldots , \lambda _\ell)$ we define the \emph{elementary symmetric function} $e_\lambda$ to be
$$e_\lambda = e_{\lambda _1}e_{\lambda _2}\cdots e_{\lambda _\ell}.$$The \emph{algebra of symmetric functions}, $\Lambda$, is then the graded algebra
$$\Lambda = \Lambda ^0 \oplus \Lambda ^1 \oplus \cdots$$where $\Lambda ^0 = \spam \{1=e_0\} = \bQ$ and for $N\geq 1$
$$\Lambda ^N = \spam \{e_\lambda \suchthat \lambda \vdash N\}.$$In fact the elementary symmetric functions form a basis for $\Lambda$ and the fundamental theorem of symmetric functions states that
$$\Lambda = \bQ [ e_1, e_2, \ldots ].$$ Perhaps the most celebrated basis of $\Lambda$ is the basis of Schur functions. For a partition $\lambda = (\lambda _1, \lambda _2, \ldots, \lambda _\ell)\vdash N$, we define the \emph{Schur function} $s_\lambda $ to be
\begin{equation}\label{eq:JT}
s_{\lambda }=\det \left( e_{\lambda ^t_i -i +j}\right) _{1\leq i,j \leq \lambda_1}
\end{equation}where if $\lambda ^t_{i}-i+j <0$ then $e _{\lambda ^t_{i}-i+j}=0$.

If a symmetric function can be written as a nonnegative linear combination of elementary symmetric functions then we say it is \emph{$e$-positive}, and likewise if a symmetric function can be written as a nonnegative linear combination of Schur functions then we say it is \emph{Schur-positive}. Although not clear from  \eqref{eq:JT}, it is a classical result that any elementary symmetric function is Schur-positive. However, in general 
little is known about the $e$-positivity and Schur-positivity of the symmetric function that will be our main object of study, the chromatic symmetric function. One well-known result in this direction is by Gasharov~\cite{Gasharov}, who gave an explicit Schur-positive formula for the chromatic symmetric function of the large class of graphs known as incomparability graphs of $(3+1)$-free posets. 

This function is dependent on a graph that is \emph{finite} and \emph{simple} and from now on we will assume that all our graphs satisfy this property. This function is also dependent on the concept of a proper colouring. Namely, given a graph, $G$, with vertex set $V$ a \emph{proper colouring} $\kappa$ of $G$ is a function
$$\kappa : V\rightarrow \{1,2,\ldots\}$$such that if $v_1, v_2 \in V$ are adjacent, then $\kappa(v_1)\neq \kappa(v_2)$. We are now ready to define the chromatic symmetric function.

\begin{definition}\cite[Definition 2.1]{Stan95}\label{def:chromsym} For a graph $G$ with vertex set $V=\{v_1, v_2, \ldots, v_N\}$ and edge set $E$, the \emph{chromatic symmetric function} is defined to be
$$X_G = \sum _\kappa x_{\kappa(v_1)}x_{\kappa(v_2)}\cdots x_{\kappa(v_N)}$$
where the sum is over all proper colourings $\kappa$ of $G$. If $G=\emptyset$ then $X_G=1$.
\end{definition}

The chromatic symmetric function is a natural generalization of the \emph{chromatic polynomial}, $\chi_G(x)$, of a graph $G$, which is the number of proper colourings  with $x$ colours. More precisely, if $X_G(1^x)$ means setting $x$ variables equal to 1 and the rest equal to 0, then we have the following result.

\begin{proposition}\cite[Proposition 2.2]{Stan95}\label{prop:chromsymtopoly} Let $G$ be a graph and $x$ be a positive integer. Then
$$\chi_G(x)=X_G(1^x).$$
\end{proposition}

\begin{example}\label{ex:chromsymtopoly} Let $K_m$ for $m\geq 1$ be the complete graph on $m$ vertices, each pair of which are adjacent. Then when we compute $X_{K_m}$, since every vertex must be coloured a different colour, and given $m$ colours this can be done in $m!$ ways, we obtain
$$X_{K_m} = m! \sum _{i_1<i_2 < \cdots < i_m} x_{i_1}x_{i_2}\cdots x_{i_m}$$and hence
$$\chi _{K_m}(x) = X_{K_m}(1^x) = m! {{x}\choose{m}} = x(x-1)(x-2)\cdots (x-(m-1)).$$
\end{example}

The chromatic polynomial satisfies the key property of \emph{deletion-contraction}, namely if $G$ is a graph with edge $\epsilon$, $G- \{\epsilon\}$ denotes $G$ with $\epsilon$ deleted and $G/\{\epsilon\}$ denotes $G$ with $\epsilon$ contracted and its vertices at either end identified, then
\begin{equation}\label{eq:delcon}\chi _G(x)= \chi_{G-\{ \epsilon\}}(x) - \chi_{G/ \{\epsilon\}}(x).\end{equation}This property is not satisfied by the chromatic symmetric function, however, recently Orellana and Scott proved the following beautiful \emph{triple-deletion} property, which will be useful to us in the next section. It is reliant on the existence of a 3-cycle, that is, 3 edges that form a triangle.

\begin{proposition}\cite[Theorem 3.1]{Orellana}\label{prop:OS} Let $G$ be a graph with edge set $E$ such that $\epsilon_1, \epsilon_2, \epsilon_3 \in E$ form a triangle. Then
$$X_G = X_{G- \{\epsilon_1\}}+X_{G- \{\epsilon_2\}}- X_{G- \{\epsilon_1, \epsilon_2\}}.$$
\end{proposition}

In fact, we can generalize this to $k$-cycles, which we call \emph{$k$-deletion}.

\begin{proposition}\label{prop:kdel} Let $G$ be a graph with edge set $E$ such that $\epsilon_1,\epsilon_2,\dots, \epsilon_{k} \in E$ form a $k$-cycle for $k\geq 3$. Then
$$\sum_{S\subseteq [k-1]}(-1)^{|S|}X_{G-\cup_{i\in S}\{\epsilon_i\}}=0.$$
\end{proposition}

\begin{proof} We assume that $G$ has a $k$-cycle for $k\geq 3$, with edges $\epsilon_1,\epsilon_2,\dots, \epsilon_{k}$ and vertices $v_1,v_2,\dots, v_{k}$ as below.
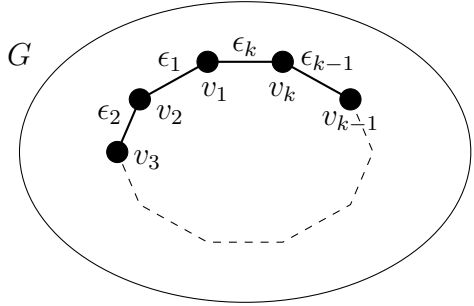
\begin{figure}[h]
\begin{center}
\begin{tikzpicture}[scale=1]
\draw (0,0) ellipse (3 and 2);
\coordinate (1) at  (-1.7,0);
\coordinate (2) at (-1.4,.7);
\coordinate (3) at (-.5,1.2);
\coordinate (4) at (.5,1.2);
\coordinate (5) at (1.4,.7);
\coordinate (6) at  (1.7,0);
\coordinate (7) at (1.4,-.7);
\coordinate (8) at  (.5,-1.2);
\coordinate (9) at  (-.5,-1.2);
\coordinate (10) at  (-1.4,-.7);
\draw[->][black, thick] (1)--(2)--(3)--(4)--(5);
\draw[->][black, dashed] (5)--(6)--(7)--(8)--(9)--(10)--(1);
\draw (0,1.4) node {$\epsilon_{k}$};
\draw (-3,1.3) node {$G$};
\draw (-1,1.2) node {$\epsilon_{1}$};
\draw (-1.8,.5) node {$\epsilon_{2}$};
\draw (1.1,1.2) node {$\epsilon_{k-1}$};
\draw (-.4,.8) node {$v_{1}$};
\draw (.5,.8) node {$v_{k}$};
\draw (-1,.5) node {$v_{2}$};
\draw (-1.3,-.1) node {$v_{3}$};
\draw (1.4,.4) node {$v_{k-1}$};
\filldraw [black] (1) circle (4pt);
\filldraw [black] (2) circle (4pt);
\filldraw [black] (3) circle (4pt);
\filldraw [black] (4) circle (4pt);
\filldraw [black] (5) circle (4pt);
\end{tikzpicture}
\end{center}
\end{figure}

We will prove the formula using a sign-reversing involution without any fixed points. Our signed set will be pairs $(\kappa, S)$ where $S\subseteq[k-1]$ and $\kappa$ is a proper colouring on the graph $G-\cup_{i\in S}\{\epsilon_i\}$. The weight of this pair will be $(-1)^{|S|}x_{\kappa}$, that is, the monomial associated to the colouring $\kappa$ with sign determined by $|S|$. 
Since the edge $\epsilon_k$ is present in all graphs $G-\cup_{i\in S}\{\epsilon_i\}$ for any $S\subseteq [k-1]$  all colourings will have at least two colours on the vertices $v_1,v_2,\dots, v_{k}$. Hence, there will always exist a smallest $j\in [k-1]$ where the two colours on $v_j$ and $v_{j+1}$ are different. We map $(\kappa, S)$ to $(\kappa, T)$ where $T=S\cup\{j\}$ if $j\notin S$ and otherwise $T = S\setminus\{j\}$. The colouring $\kappa$ is a proper colouring on $G-\cup_{i\in S}\{\epsilon_i\}$ and $G-\cup_{i\in T}\{\epsilon_i\}$ since we are only including or excluding the edge $\epsilon_j$, which has adjacent vertices $v_j$ and $v_{j+1}$ with different colours. We can easily see that this is an involution that only switches the sign of the weight, and since there are no  fixed points our desired sum is zero. 
\end{proof}

\section{The chromatic symmetric function of lollipops}\label{sec:lollipop} We now come to our object of study, the chromatic symmetric function of lollipop graphs, or \emph{lollipop symmetric functions}. However, before we define them we recall some necessary graphs, whose definitions we give for clarity. The \emph{complete graph} $K_m$ for $m\geq 1$ has $m$ vertices, each pair of which are adjacent. The \emph{path graph} $P_n$ for $n\geq 1$ has $n$ vertices and is the tree with 2 vertices of degree 1 and $n-2$ vertices of degree 2 for $n\geq 2$ and $P_1=K_1$. From these we obtain the \emph{lollipop graph} $L_{m,n}$ for $m\geq 1, n\geq 1$, which we construct from the disjoint union of $K_m$ and $P_n$ by connecting with an edge one vertex of $K_m$ and one vertex of degree 1 in $P_n$. For example, $L_{3,6}$ is below.

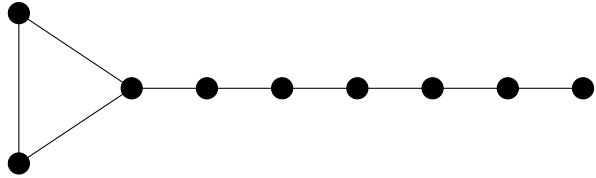
\begin{figure}[h]
\begin{center}
\begin{tikzpicture}
\filldraw [black] (0,0) circle (4pt);
\filldraw [black] (1.5,1) circle (4pt);
\filldraw [black] (0,2) circle (4pt);
\filldraw [black] (2.5,1) circle (4pt);
\filldraw [black] (3.5,1) circle (4pt);
\filldraw [black] (4.5,1) circle (4pt);
\filldraw [black] (5.5,1) circle (4pt);
\filldraw [black] (6.5,1) circle (4pt);
\filldraw [black] (7.5,1) circle (4pt);
\draw (0,0)--(1.5,1)--(0,2)--(0,0);
\draw (1.5,1)--(7.5,1);
\end{tikzpicture}
\end{center}
\end{figure}

For convenience we define 
$$K_0=P_0=L_{0,0}=\emptyset$$namely, the empty graph, and we define  $L_{m,0}=K_m$ and $L_{0,n}=P_n$. We are now ready to define lollipop symmetric functions.

\begin{definition}\label{def:lolli} Let $m\geq 0$ and $ n\geq 0$. Then the \emph{lollipop symmetric function} is given by $X_{L_{m,n}}$, that is, the chromatic symmetric function of $L_{m,n}$.
\end{definition}

By definition $X_{L_{m,0}}=X_{K_m}$ and $X_{L_{0,n}}=X_{P_n}$  both of which are $e$-positive, and hence Schur-positive. More precisely by, for example, \cite[Theorem 8]{ChovW}
\begin{equation}\label{eq:Km}
X_{K_m}=m! e_m
\end{equation}for $m\geq 1$, and by \cite[Theorem 3.2]{Wolfe} the coefficient of $e_1^{a_1}e_2^{a_2}\cdots e_n^{a_n}$ in $X_{P_n}$ for $n\geq 1$ is
\begin{equation}\label{eq:Pn}
{{a_1+\cdots + a_n} \choose {a_1, \ldots , a_n}} \prod _{j=1} ^n (j-1)^{a_j} + \sum _{i\geq 1}\left( {{(a_1+\cdots + a_n)-1}\choose {a_1, \ldots , a_i-1 , \ldots , a_n}} \left( \prod _{j=1, j\neq i} ^n (j-1)^{a_j}\right) (i-1)^{a_i-1}\right).
\end{equation}Hence a natural avenue to pursue is whether lollipop symmetric functions are $e$-positive, and hence Schur-positive, and if so what is an explicit combinatorial $e$-positive formula. The former was done implicitly in \cite[Corollary 7.7]{GebSag}, and we will now prove the former explicitly and give the latter in Proposition~\ref{prop:lolliKm}. In order to do this we need the following observation.

Note that for our results we will often have the bound $m\geq 2$, however, this does not restrict our results and is simply for clarity of exposition. This is because for $n\geq 0$
$$L_{2,n}=L_{1,n+1} = L_{0,n+2}=P_{n+2}.$$Also $L_{1,0}=L_{0,1}=K_1=P_1$ and we know that the chromatic symmetric function in this case is
$$x_1+x_2+\cdots = e_1=s_1.$$Finally, by definition, the chromatic symmetric function of $L_{0,0}$ is $1=e_0=s_0$.

\begin{theorem}\label{the:X_{L_{m,n}}} For $m\geq 2$ and $n\geq 0$ we have that
$$X_{L_{m,n}}=(m-1)X_{L_{m-1,n+1}} - (m-2) X_{K_{m-1}}X_{P_{n+1}}.$$
\end{theorem}

\begin{proof} We will first focus on $L_{m,n}$ in order to deduce our eventual result on $X_{L_{m,n}}$. First let $m=2$. Then $L_{2,n} = L_{1,n+1}$ and hence
$$X_{L_{2,n}}=X_{L_{1,n+1}}$$as desired. 

Now let $m\geq 3$ and consider $L_{m,n}$. This graph has triangles, so we will name some edges in order to identify the triangles we want to focus on. There is a unique vertex in our graph of degree $m$ where $m-1$ edges belong to the copy of $K_m$ and one edge is the bridge to the copy of $P_n$. Label these $m-1$ edges $\epsilon _1 , \epsilon _2, \ldots , \epsilon _{m-1}$. Also let $f_i$ be the edge, which together with $\epsilon _i$ and $\epsilon _{i+1}$ form a triangle. See the illustrative diagram below.

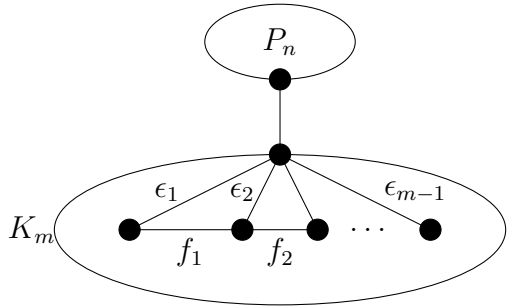
\begin{figure}[h]
\begin{center}
\begin{tikzpicture}
\draw (0,1.5) ellipse (1 and .5);
\draw (0,-1) ellipse (3 and 1);
\filldraw [black] (0,1) circle (4pt);
\filldraw [black] (0,0) circle (4pt);
\filldraw [black] (-2,-1) circle (4pt);
\filldraw [black] (-.5,-1) circle (4pt);
\filldraw [black] (2,-1) circle (4pt);
\filldraw [black] (0.5,-1) circle (4pt);
\draw (0,0)--(-2,-1);
\draw (-2,-1)--(0.5,-1);
\draw (0,0)--(-.5,-1);
\draw (0,0)--(2,-1);
\draw (0,0)--(0,1);
\draw (0,0)--(0.5,-1);
\draw (0,1.5) node {$P_n$};
\draw (-3.3,-1) node {$K_m$};
\draw (1.2,-1) node {$\dots$};
\draw (-.5,-.5) node {$\epsilon_2$};
\draw (1.8,-.5) node {$\epsilon_{m-1}$};
\draw (-1.5,-.5) node {$\epsilon_1$};
\draw (-1.2,-1.3) node {$f_1$};
\draw (0,-1.3) node {$f_2$};
\end{tikzpicture}
\end{center}
\end{figure}

Note that if we remove any subsets $S\subseteq \{\epsilon_1, \epsilon_2, \ldots , \epsilon_{m-1}\}$ of equal size $k$ from $L_{m,n}$ then these graphs are isomorphic, and hence without loss of generality we will focus on $L_{m,n}-S_k$ where 
$S_k = \{\epsilon _1, \epsilon _2, \ldots , \epsilon _k\}$. Three cases of note are as follows. First, if we remove $S_0$, then by removing no edges $L_{m,n}$ remains unchanged, that is, $L_{m,n} - S_0 = L_{m,n}$. Second, if we remove $S_{m-1}$, then by removing all $m-1$ 
edges we separate $L_{m,n}$ into two connected components, one being $K_{m-1}$ and the other being $P_{n+1}$, that is, $L_{m,n}-S_{m-1} = K_{m-1}\cup P_{n+1}$. Third, if we remove $S_{m-2}$, then by removing all edges $\epsilon _i$ except $\epsilon _{m-1}$, we obtain $L_{m-1, n+1}$, that is, $L_{m,n}-S_{m-2} = L_{m-1,n+1}$.

Now let us apply Proposition~\ref{prop:OS} to $X_{L_{m,n}-S_{k-1}}$ for $1\leq k \leq m-2$ using the triangle
 with edges $\epsilon_k$, $\epsilon_{k+1}$, and $f_k$. We then get
$$X_{L_{m,n}-S_{k-1}}=X_{L_{m,n}-S_{k-1}-\{\epsilon_k\}}+X_{L_{m,n}-S_{k-1}-\{\epsilon_{k+1}\}}-X_{L_{m,n}-S_{k-1}-\{\epsilon_k,\epsilon_{k+1}\}}=2X_{L_{m,n}-S_{k}}-X_{L_{m,n}-S_{k+1}}.$$
If we apply this continually to $X_{L_{m,n}}=X_{L_{m,n}-S_0}$  we get
\begin{align*}
X_{L_{m,n}}&=2X_{L_{m,n}-S_{1}}-X_{L_{m,n}-S_{2}}\\
&=2(2X_{L_{m,n}-S_{2}}-X_{L_{m,n}-S_{3}})-X_{L_{m,n}-S_{2}}\\
&=3X_{L_{m,n}-S_{2}}-2X_{L_{m,n}-S_{3}}\\
&=3(2X_{L_{m,n}-S_{3}}-X_{L_{m,n}-S_{4}})-2X_{L_{m,n}-S_{3}}\\
&=4X_{L_{m,n}-S_{3}}-3X_{L_{m,n}-S_{4}}\\
&\vdots\\
&= kX_{L_{m,n}-S_{k-1}}-(k-1)X_{L_{m,n}-S_{k}}\\
&=k(2X_{L_{m,n}-S_{k}}-X_{L_{m,n}-S_{k+1}})-(k-1)X_{L_{m,n}-S_{k}}\\
&=(k+1)X_{L_{m,n}-S_{k}}-kX_{L_{m,n}-S_{k+1}}
\end{align*}
for $1\leq k\leq m-2$. In particular, at $k=m-2$  we get
\begin{align*}
X_{L_{m,n}}&=(m-1)X_{L_{m,n}-S_{m-2}}-(m-2)X_{L_{m,n}-S_{m-1}}\\
&=(m-1)X_{L_{m-1,n+1}}-(m-2)X_{K_{m-1}}X_{P_{n+1}}
\end{align*}
and we are done.
\end{proof}

We can now affirm directly the $e$-positivity, and hence Schur-positivity, of lollipop symmetric functions, which was proved indirectly in \cite[Corollary 7.7]{GebSag}.

\begin{theorem}\label{the:epos} For all $m\geq0$ and $n\geq 0$ we have that $X_{L_{m,n}}$ is $e$-positive, and hence Schur-positive.
\end{theorem}

\begin{proof} First note that when $m=0$ we have that $X_{L_{m,n}}= X_{P_n}$, which we know is $e$-positive by \eqref{eq:Pn} and Definition~\ref{def:chromsym}. We will now prove  the result by induction on $n$. When $n=0$ we have $X_{L_{m,0}}=X_{K_m}$, which we know is $e$-positive by \eqref{eq:Km} and Definition~\ref{def:chromsym}. Now assume that for any $m\geq 1$ and $0<k<n$ that $X_{L_{m,k}}$ is $e$-positive.

By Theorem~\ref{the:X_{L_{m,n}}} for $n\geq 0$ and $m\geq 2$ we have that 
\begin{align*}
X_{L_{m,n}}&=(m-1)X_{L_{m-1,n+1}}-(m-2)X_{K_{m-1}}X_{P_{n+1}}\\
\Rightarrow X_{L_{m,n}}+(m-2)X_{K_{m-1}}X_{P_{n+1}}&=(m-1)X_{L_{m-1,n+1}}.
\end{align*}
By shifting the indices we get for $m\geq 1$ and $n\geq 1$ that
\begin{align*}
mX_{L_{m,n}}&=X_{L_{m+1,n-1}}+(m-1)X_{K_{m}}X_{P_{n}}\\
\Rightarrow X_{L_{m,n}}&=\frac{1}{m}\left(X_{L_{m+1,n-1}}+(m-1)X_{K_{m}}X_{P_{n}}\right).
\end{align*}

By our inductive assumption we know that $X_{L_{m+1,n-1}}$ is $e$-positive. Plus, as stated earlier in the proof, $X_{K_m}$ and $X_{P_n}$ for all $m,n\geq 1$ are $e$-positive, and since elementary symmetric functions are multiplicative $X_{K_m}X_{P_n}$ is $e$-positive. Since we have now shown that $X_{L_{m,n}}$ is a linear combination of two $e$-positive symmetric functions, we can conclude that $X_{L_{m,n}}$ is $e$-positive, and hence Schur-positive. \end{proof}

A natural question to ask is whether succinct formulas exist for $X_{L_{m,n}}$ in terms of $X_{K_m}$ and $X_{P_n}$ since $K_m$ and $P_n$ are required to construct $L_{m,n}$. Such formulas are obtainable using Theorem~\ref{the:X_{L_{m,n}}} and the proof of Theorem~\ref{the:epos}, from which the expansion into elementary symmetric functions is immediate using \eqref{eq:Km} and \eqref{eq:Pn}.

\begin{proposition}\label{prop:lolliPn} For $m\geq 2$ and $n\geq 0$ we have that
$$X_{L_{m,n}}= (m-1)! \left( X_{P_{n+m}} - \sum _{i=1} ^{m-2} \frac{(m-i-1)}{(m-i)!} X_{K_{m-i}}X_{P_{n+i}}\right).$$
\end{proposition}

\begin{proof} We prove this by induction on $m$ for $m\geq 2$. When $m=2$ and $n\geq 0$ we have that
$$X_{L_{2,n}}=X_{P_{n+2}}$$as desired. Now assume that the result holds for all $2<k<m$ and $n\geq 0$. Then by Theorem~\ref{the:X_{L_{m,n}}} we have that
\begin{align*}
X_{L_{m,n}} =& (m-1) X_{L_{m-1, n+1}} - (m-2) X_{K_{m-1}}X_{P_{n+1}}\\
=& (m-1) \left( (m-2)! \left( X_{P_{n+m}} - \sum _{i=1}^{m-3} \frac{(m-i-2)}{(m-i-1)!} X_{K_{m-i-1}}X_{P_{n+i+1}}\right)\right) \\&- (m-2) X_{K_{m-1}}X_{P_{n+1}}\\
=&(m-1)! \left( X_{P_{n+m}} - \sum _{i=1} ^{m-2} \frac{(m-i-1)}{(m-i)!} X_{K_{m-i}}X_{P_{n+i}}\right)
\end{align*}by induction.
\end{proof}

This first formula will play a key role in the next section, while the second formula below yields an explicit $e$-positive formula for $X_{L_{m,n}}$ using \eqref{eq:Km} and \eqref{eq:Pn}.

\begin{proposition}\label{prop:lolliKm} For $m\geq 2$ and $n\geq 0$ we have that
$$X_{L_{m,n}}= \frac{(m-1)!}{ (m+n-1)!} X_{K_{m+n}} + \sum _{i=0} ^{n-1} \frac{(m+i-1)}{m(m+1)\cdots (m+i)} X_{K_{m+i}}X_{P_{n-i}}.$$
\end{proposition}

\begin{proof} We prove this by induction on $n$ for $n\geq0$. When $n=0$ and $m\geq 2$ we have that
$$X_{L_{m,0}}=X_{K_{m}}$$as desired. Now assume that the result holds for all $0<k<n$ and $m\geq 2$. Then by the last equation in the proof of Theorem~\ref{the:epos} we have that
\begin{align*}
X_{L_{m,n}} =& \frac{1}{m}\left( X_{L_{m+1,n-1}} + (m-1) X_{K_m}X_{P_n}\right)\\
=&\frac{1}{m}\left(  \frac{m!}{(m+n-1)!} X_{K_{m+n}} \right.\\
&\left. + \sum _{i=0} ^{n-2} \frac{(m+i)}{(m+1)\cdots (m+i+1)} X_{K_{m+i+1}}X_{P_{n-i-1}}+ (m-1) X_{K_m}X_{P_n}\right)\\
=&\frac{(m-1)!}{(m+n-1)!} X_{K_{m+n}} + \sum _{i=0} ^{n-1} \frac{(m+i-1)}{m(m+1)\cdots (m+i)} X_{K_{m+i}}X_{P_{n-i}}
\end{align*}by induction.
\end{proof}

\begin{lemma}\label{lem:chrompoly} For $m\geq 2$ and $n\geq 0$ we have that
$$\chi _{L_{m,n}} (x) = x(x-1)^{n+1}(x-2) \cdots (x-(m-1)).$$
\end{lemma}

\begin{proof} We prove this by induction on $n$. For $n=0$ and $m\geq 2$ we have that $L_{m,n}=K_m$ whose chromatic polynomial, which we computed in Example~\ref{ex:chromsymtopoly}, is well known to be 
$$x(x-1)(x-2) \cdots (x-(m-1)).$$Now assume that the result holds for all $0<k<n$ and $m\geq 2$. Then by deletion-contraction where we delete an edge incident to  a degree 1 vertex
$$\chi _{L_{m,n}} (x) = (x-1)\chi _{L_{m,n-1}}= (x-1)x(x-1)^{n}(x-2) \cdots (x-(m-1))$$as desired.
\end{proof}

\begin{lemma}\label{lem:distinct} If $L_{m,n}$ and $L_{m',n'}$ satisfy $(m,n)\neq (m',n')$ where $m,m'\geq2$ and $ n, n'\geq 0$, then
$$X_{L_{m,n}}\neq X_{L_{m',n'}}.$$
\end{lemma}

\begin{proof} First observe that by Lemma~\ref{lem:chrompoly} that since $m$ and $n$ are recoverable from the chromatic polynomial, if $(m,n)\neq (m',n')$ then $\chi _{L_{m,n}}(x) \neq \chi _{L_{m',n'}}(x) $. Furthermore by Proposition~\ref{prop:chromsymtopoly} the chromatic symmetric function reduces to the chromatic polynomial and hence if $(m,n)\neq (m',n')$ then
$X_{L_{m,n}}\neq X_{L_{m',n'}}.$
\end{proof}

We are now ready to prove that there exist countably infinite distinct bases arising from chromatic symmetric functions that are $e$-positive, and hence Schur-positive.

\begin{theorem}\label{the:infbases} There exist countably infinite distinct $e$-positive, and hence Schur-positive, bases of the algebra of symmetric functions $\Lambda$ whose generators are the chromatic symmetric functions of connected graphs. In particular, every distinct set of lollipops $\{ {\mathcal L}_{1} , {\mathcal L}_{2}, \ldots \}$ where ${\mathcal L}_i=L_{m_i,n_i}$ for some $m_i+n_i = i$ gives rise to a distinct set of generators $\{ X_{\mathcal{L}_1}, X_{\mathcal{L}_2}, \ldots \}$ such that $$\Lambda = \mathbb{Q}[ X_{\mathcal{L}_1}, X_{\mathcal{L}_2}, \ldots ].$$
\end{theorem}

\begin{proof} To begin, first recall that $L_{1,0}=L_{0,1}=K_1=P_1$, plus $L_{2,0}=L_{1,1}=L_{0,2}=K_2=P_2$, and for a given integer $k\geq 3$ there exist $k-1$ distinct lollipop graphs
$$L_{k,0}, \ldots , L_{2, k-2}.$$ Hence there exist countably infinite distinct sets of lollipop graphs 
$$\{ L_{m_1, n_1} , L_{m_2, n_2}, \ldots \}$$where $m_i+n_i = i$. Given such a set $\mathcal{L}$, denote $L_{m_i, n_i}$ by $\mathcal{L}_i$. Then by \cite[Theorem 5]{ChovW}
$$\{ X_{\mathcal{L}_{\lambda _1}}X_{\mathcal{L}_{\lambda _2}}\cdots X_{\mathcal{L}_{\lambda _\ell}} \suchthat (\lambda _1, \lambda _2, \ldots ,\lambda _\ell )\vdash N\}$$is a $\mathbb{Q}$-basis for $\Lambda ^N$ and
$$\Lambda = \mathbb{Q}[ X_{\mathcal{L}_1}, X_{\mathcal{L}_2}, \ldots ].$$Moreover, since each $X_{\mathcal{L}_i}$ is $e$-positive by Theorem~\ref{the:epos} and elementary symmetric functions are multiplicative, it follows that each
$$X_{\mathcal{L}_{\lambda _1}}X_{\mathcal{L}_{\lambda _2}}\cdots X_{\mathcal{L}_{\lambda _\ell}} $$above is $e$-positive, and hence Schur-positive. By Lemma~\ref{lem:distinct} it follows that every set $\mathcal{L}$ yields a distinct set of generators $\{ X_{\mathcal{L}_1}, X_{\mathcal{L}_2}, \ldots \}$ for $\Lambda$ and we are done.
\end{proof}

\section{The chromatic symmetric function of lariats}\label{sec:lariat}

In this final section we resolve 6 conjectures from 1998 \cite[p 656]{Wolfe} regarding the chromatic symmetric function of lariat graphs, where the \emph{lariat graph} $L_{n+3}$ for $n\geq 0$ is defined to be
$$L_{n+3}=L_{3,n}.$$More precisely we prove 5 of the conjectures in Theorem~\ref{the:lariat} and conclude with a counterexample to the sixth.

Before we do, we note that setting $m=3$ in Proposition~\ref{prop:lolliPn} yields the following expression for the chromatic symmetric function of $L_{n+3}$ for $n\geq0$
\begin{equation}\label{eq:lariat}
X_{L_{n+3}}= 2\left( X_{P_{n+3}} - \frac{1}{2}X_{K_2}X_{P_{n+1}}\right) = 2 X_{P_{n+3}} - 2e_2X_{P_{n+1}}
\end{equation}since $X_{K_2}=2e_2$ by \eqref{eq:Km}. We are now ready to prove our theorem.

\begin{theorem}\label{the:lariat} Consider $X_{L_{n+3}}$ for $n\geq 0$ expanded into the basis of elementary symmetric functions.
\begin{enumerate}
\item If $e_\lambda$ appears in $X_{L_{n+3}}$ with nonzero coefficient, then $e_\lambda$ appears in $X_{P_{n+3}}$ with nonzero coefficient.
\item Let $\lambda$ be a partition with no part equal to 2. Then if $e_\lambda$ appears in $X_{P_{n+3}}$ with coefficient $c_\lambda$, then $e_\lambda$ appears in $X_{L_{n+3}}$ with coefficient $2c_\lambda$.
\item The coefficient of $e_{(n+1,2)}$ in $X_{L_{n+3}}$ is $4n$.
\item The coefficient of $e_{(n,2,1)}$ in $X_{L_{n+3}}$ is $2(n-1)$.
\item The coefficient of $e_{(n-1,2,2)}$ in $X_{L_{n+3}}$ is $4(n-2)$.
\end{enumerate}
\end{theorem}

\begin{proof} The first part follows from \eqref{eq:lariat} and the fact that $X_{L_{n+3}}$ is $e$-positive by Theorem~\ref{the:epos} and $X_{P_{n+3}}$ is $e$-positive by \eqref{eq:Pn}. The second part follows immediately from \eqref{eq:lariat}. For the third, fourth and fifth parts, respectively, we again use \eqref{eq:lariat},
$$X_{L_{n+3}} = 2 X_{P_{n+3}} - 2e_2X_{P_{n+1}},$$in conjunction with \eqref{eq:Pn} to obtain the following coefficients.
\begin{itemize}
\item To compute the coefficient of $e_{(n+1,2)}$ in $X_{L_{n+3}}$, note that its coefficient in $2X_{P_{n+3}}$ is
$$2(2(n+1-1)+(n+1-1)+1)$$whereas in $2e_2X_{P_{n+1}}$ it is $2((n+1-1)+1)$.
\item To compute the coefficient of $e_{(n,2,1)}$ in $X_{L_{n+3}}$, note that its coefficient in $2X_{P_{n+3}}$ is
$$2(2(n-1))$$whereas in $2e_2X_{P_{n+1}}$ it is $2(n-1)$.
\item To compute the coefficient of $e_{(n-1,2,2)}$ in $X_{L_{n+3}}$, note that its coefficient in $2X_{P_{n+3}}$ is
$$2(3(n-1-1)+2(n-1-1)+1)$$whereas in $2e_2X_{P_{n+1}}$ it is $2(2(n-1-1)+(n-1-1)+1)$.
\end{itemize}
\end{proof}

The final conjecture stated that if $e_\lambda$, where $\lambda = (1^{a_1}, 2^{a_2}, \ldots , (n+3)^{a_{n+3}})$, appears with nonzero coefficient in $X_{L_{n+3}}$ when expanded into the basis of elementary symmetric functions, then $a_i \leq 2$ for all $i$. This is false since $L_9=L_{3,6}$, which is the smallest counterexample and is shown earlier in Section~\ref{sec:lollipop}, has chromatic symmetric function
\begin{align*}
X_{L_9}&= 8e_{(3,2,2,2)}+16 e_{(3,3,2,1)}+24e_{(3,3,3)}+6e_{(4,2,2,1)}+82e_{(4,3,2)}\\
&+18e_{(4,4,1)}+16e_{(5,2,2)}+32e_{(5,3,1)}+62e_{(5,4)}+10e_{(6,2,1)}\\
&+54e_{(6,3)}+24e_{(7,2)}+14e_{(8,1)}+18e_{(9)}.
\end{align*}

\section*{Acknowledgements}\label{sec:acknow} The authors would like to thank Ang\`{e}le Hamel and Foster Tom for conversations that sparked fruitful avenues of research, and Richard Stanley and the referees for helpful comments.

\bibliographystyle{siamplain}

\def\cprime{$'$}

\end{document}